\documentclass[11pt, a4paper]{article}
\usepackage[utf8]{inputenc}
\usepackage{mathtools}
\usepackage{amsmath}
\usepackage{amsthm}
\usepackage{algorithm}
\usepackage{algorithmic}
\usepackage{amssymb}
\usepackage{epsfig}
\usepackage{titlesec}
\allowdisplaybreaks

\usepackage{stackengine}
\mathtoolsset{showonlyrefs}
\usepackage{xcolor}
\usepackage[a4paper, margin=2.5cm]{geometry}
\usepackage{comment}
\usepackage{mathrsfs} 
\usepackage[export]{adjustbox}

\usepackage{multirow}

\usepackage{dsfont}
\usepackage{fixltx2e}
\usepackage{float}
\usepackage{mathrsfs}
\usepackage{amssymb,amsmath,amsthm,amsfonts}
\usepackage{mathtools}
\usepackage{fixmath}
\usepackage{bbm}
\usepackage{braket}
\usepackage{caption}
\usepackage{subcaption}
\usepackage{enumitem}
\usepackage{units}
\usepackage{xcolor}
\usepackage{booktabs}
\usepackage[english]{babel}
\usepackage[T1]{fontenc}
\usepackage{lmodern}
\usepackage{tikz}
\usepackage{pgfplots}
\usepackage{algorithm}
\usepackage{algorithmic}
\usepackage{comment}

\usepackage[final,stretch=10]{microtype}

\usepackage[colorlinks, citecolor=hanblue,linkcolor=black]{hyperref}
\definecolor{hanblue}{rgb}{0.27, 0.42, 0.81}
\usepackage{todonotes}

\newcommand{\tr}[1]{tr(C(1)^{-1})}

\newcommand{\vertiii}[1]{{\left\vert\kern-0.25ex\left\vert\kern-0.25ex\left\vert #1
		\right\vert\kern-0.25ex\right\vert\kern-0.25ex\right\vert}}

\makeatletter
\DeclareFontFamily{OMX}{MnSymbolE}{}
\DeclareSymbolFont{MnLargeSymbols}{OMX}{MnSymbolE}{m}{n}
\SetSymbolFont{MnLargeSymbols}{bold}{OMX}{MnSymbolE}{b}{n}
\DeclareFontShape{OMX}{MnSymbolE}{m}{n}{
	<-6>  MnSymbolE5
	<6-7>  MnSymbolE6
	<7-8>  MnSymbolE7
	<8-9>  MnSymbolE8
	<9-10> MnSymbolE9
	<10-12> MnSymbolE10
	<12->   MnSymbolE12
}{}
\DeclareFontShape{OMX}{MnSymbolE}{b}{n}{
	<-6>  MnSymbolE-Bold5
	<6-7>  MnSymbolE-Bold6
	<7-8>  MnSymbolE-Bold7
	<8-9>  MnSymbolE-Bold8
	<9-10> MnSymbolE-Bold9
	<10-12> MnSymbolE-Bold10
	<12->   MnSymbolE-Bold12
}{}

\let\llangle\@undefined
\let\rrangle\@undefined
\DeclareMathDelimiter{\llangle}{\mathopen}%
{MnLargeSymbols}{'164}{MnLargeSymbols}{'164}
\DeclareMathDelimiter{\rrangle}{\mathclose}%
{MnLargeSymbols}{'171}{MnLargeSymbols}{'171}
\makeatother

\makeatletter
\newcommand*\rel@kern[1]{\kern#1\dimexpr\macc@kerna}
\newcommand*\widebar[1]{%
	\begingroup
	\def\mathaccent##1##2{%
		\rel@kern{0.8}%
		\overline{\rel@kern{-0.8}\macc@nucleus\rel@kern{0.2}}%
		\rel@kern{-0.2}%
	}%
	\macc@depth\@ne
	\let\math@bgroup\@empty \let\math@egroup\macc@set@skewchar
	\mathsurround\z@ \frozen@everymath{\mathgroup\macc@group\relax}%
	\macc@set@skewchar\relax
	\let\mathaccentV\macc@nested@a
	\macc@nested@a\relax111{#1}%
	\endgroup
}
\makeatother

\numberwithin{equation}{section}

\definecolor{darkred}{rgb}{.7,0,0}

\definecolor{green}{rgb}{0,0.7,0}

\usepackage[notref,notcite]{}
\theoremstyle{plain}

\newtheorem{theorem}{Theorem}[section]

\newtheorem{proposition}[theorem]{Proposition}

\theoremstyle{definition}

\newtheorem{defs}[theorem]{Definition}

\newtheorem{rem}[theorem]{Remark}

\usepackage{scalerel}

\title{Exact Sparse Representation Recovery in Signal Demixing and Group BLASSO
}
\author{Marcello Carioni \!\!\thanks{Department of Applied Mathematics, University of Twente, 7500AE Enschede, The Netherlands \\
(\texttt{m.c.carioni@utwente.nl}, \texttt{l.delgrande{@}utwente.nl})} , Leonardo Del Grande\footnotemark[1]
}
\date{}

\begin{document}
 
	
	

	\maketitle
\begin{abstract}
\noindent
In this short article we present the theory of sparse representations recovery in convex regularized optimization problems introduced in \cite{carioni2023general}. We focus on the scenario where the unknowns belong to Banach spaces and measurements are taken in Hilbert spaces, exploring the properties of minimizers of optimization problems in such settings. Specifically, we analyze a Tikhonov-regularized convex optimization problem, where $y_0$
  are the measured data, $w$ denotes the noise, and $\lambda$ is the regularization parameter. By introducing a Metric Non-Degenerate Source Condition (MNDSC) and considering sufficiently small $\lambda$ and $w$, we establish Exact Sparse Representation Recovery (ESRR) for our problems, meaning that the minimizer is unique and precisely recovers the sparse representation of the original data. We  then emphasize the practical implications of this theoretical result through two novel applications: signal demixing and super-resolution with Group BLASSO.  
These applications underscore the broad applicability and significance of our result, showcasing its potential across different domains.
\end{abstract}
\maketitle                   

\section{Introduction}
In \cite{carioni2023general}, we introduced a general theory for the recovery of the sparse representation of data in infinite-dimensional convex optimization problems. More precisely, given a Banach space $X$, a linear operator $K:X\rightarrow Y$ mapping to an Hilbert space $Y$, and a convex regularizer $R$, we  analyzed  the sparsity for minimizers of the following Tikhonov-regularized optimization problem:
\begin{equation}\label{eq:ip}
    \inf_{u\in X} \frac{1}{2}\|Ku-y_0 - w\|_Y^2+ \lambda R(u) \tag{$\mathcal{P}_{\lambda}(y_0+w)$},
\end{equation}
where $\lambda > 0$ is the regularization parameter and $w \in Y$ is the noise assumed to be deterministic. In the main result of \cite{carioni2023general}, we showed that under a certain condition on the sparse solution $u_0 \in X$ to the hard-constrained problem without noise:
\begin{equation} \label{int1}
	\inf_{u\in X: Ku=y_0} R(u)
 \tag{$\mathcal{P}_{h}(y_0)$}, 
\end{equation} 
the minimizer for \ref{eq:ip} is unique and precisely recovers the sparse representation of $u_0$ in case of small regularization parameter and in a low-noise regime, achieving \emph{Exact Sparse Representation Recovery (ESRR)}. 
ESRR provides important insights on the \emph{sparse stability} of convex regularized optimization problems when measurements are corrupted by noise. Moreover, it positions itself in the emerging theory of infinite-dimensional sparsity, providing a general recipe for studying the sparse representation of minimizers in specific optimization problems that are relevant for applications. 
\\
Prior to \cite{carioni2023general}, ESRR had been only established for optimization problems in the space of measures regularized with the total variation norm \cite{duval2015exact}. Other than that very few results are available and they only address few specific optimization problems \cite{de2023exact}.
The theory introduced in \cite{carioni2023general} is the first step towards a comprehensive theory for the exact sparse representation recovery in general infinite-dimensional convex optimization problems. 
\\
In this short paper we will summarize the approach of \cite{carioni2023general} to ESRR, highlighting the important ideas behind it. In particular, we will put emphasis on the suitable non-degeneracy assumption on the solution to \ref{int1} that is necessary to ensure ESRR. Such condition, named Metric Non-Degenerate Source Condition (MNDSC) is based on the geometric structure of the regularizer $R$ and extends the classical Non-Degenerate Source Condition (NDSC) for optimization problems in the space of measures regularized with total variation norm \cite{duval2015exact}.
We complement the paper by proposing two novel applications of our general theory: signal demixing and super-resolution with Group BLASSO regularization. Such examples are instrumental to highlight the applicability potential of   \cite{carioni2023general}, showing that many problems can be studied using the general framework introduced in \cite{carioni2023general}.

\section{Convex regularized optimization problems}

Convex regularized optimization problems are often motivated as Tikhonov-type approach to inverse problems. 
Given a linear weak*-to-weak operator $K : X\rightarrow Y$ from a Banach space $X$ to an Hilbert space $Y$, a classical inverse problem task is to reconstruct a data $u \in X$ from measurements $y \in Y$ modelled as
\begin{align}
    y = Ku + w,
\end{align}
where $w \in Y$ is the noise level, typically assumed as deterministic. Since $K$ is often non-injective and ill-conditioned, constructing an inverse of $K$ that is stable with respect to  noise is problematic. Therefore, the classical approach of Tikhonov-regularized inverse problem is to solve \ref{eq:ip},
where $R : X \rightarrow [0,+\infty]$ is a suitable functional, called regularizer, that is responsible both for increasing the stability of the optimization problem and selecting desired solutions of the inverse problem. In our scenario, $R$ is a convex, weak* lower semi-continuous and positively 1-homogeneous functional, i.e. $R(\lambda u)=\lambda R(u)$ for every $\lambda \geqslant 0$. Moreover, the sublevel set
$
S^{-}(R, \alpha):=\{u \in X: R(u) \leqslant \alpha\}
$
is weak* compact for every $\alpha \geqslant 0$ and 
$0$ is an interior point of $\partial R(0)$, where $\partial R$ denotes the subdifferential of $R$.

\subsection{Sparsity enforcing regularizers}

It has been shown that the regularizer $R$ could be chosen to promote sparse solutions to \ref{eq:ip}. In particular, if $X$ is finite dimensional and $R$ is chosen to be the $1$-norm, then solutions to \ref{eq:ip}  made of vectors with few non-zero entries are favoured. This phenomena has been thoroughly studied in compressed sensing theory \cite{candes2006robust}. If $X$ is the space of Radon measures $M(\Omega)$, where $\Omega$ is a closed subset of $\mathbb{R}^d$,  and $R$ is chosen to be the total variation norm, then it has been observed that \ref{eq:ip} promotes solutions that are linear combinations of Dirac deltas: 
\begin{align}
    \mu = \sum_{i=1}^n c_i \delta_{x_i} \quad c_i \in \mathbb{R}, \quad x_i \in \Omega.
\end{align}
As pointed out in the introduction, a suitable concept of sparsity was recently introduced for general convex regularizers $R : X \rightarrow [0,\infty]$ defined on the Banach space $X$. In particular, in \cite{boyer2019representer,bredies2020sparsity} it has been shown that sparse solutions can be naturally defined as linear combinations of extreme points of the unit ball of $R$, denoted by $B = \{u \in X : R(u) \leqslant1\}$, meaning elements of the Banach space $X$ that admits the representation
\begin{align}\label{eq:repre}
    u = \sum_{i=1}^n c_iu_i\quad c_i\in \mathbb{R},\ \  u_i\in \operatorname{Ext}(B),
\end{align}
where $\operatorname{Ext}(B)$ denotes the set of extreme points of $B$. 
This notion of sparse solutions has been justified by showing that, in case the measurement operator $K$ maps to a finite-dimensional space, then it is always possible to find a sparse solution for \ref{eq:ip}. These results are known as \emph{representer theorems}. 
Despite the importance of the representer theorem, in many situations, knowing the existence of a sparse solution is not enough since, due to the non-injectivity of $K$ and the lack of strict convexity of $R$, \ref{eq:ip} has typically many minimizers. Therefore, it is natural to ask if, under suitable assumptions, the solution to \ref{eq:ip} can be represented uniquely as in \eqref{eq:repre} at least for small values of $\lambda$ and in a low noise regime. If this holds, then we say that the \emph{Exact Sparse Representation Recovery} (ESRR) holds.

\section{Metric Non-Degenerate Source Condition}
\subsection{Duality and optimality conditions}
Before introducing the Metric Non-Degenerate Source Condition we recall the definition of a fundamental object that we will extensively use in the following: the so called \emph{dual certificate}.
First, note that the Fenchel dual problem associated with \ref{int1} reads as
\begin{align}\label{eq:dual}
    \sup _{p \in Y: K_* p \in \partial R(0)}\left(y_0, p\right)  \tag{$\mathcal{D}_{h}(y_0)$}
\end{align}
and, under the assumptions on $R$ and $K$ mentioned above, strong duality between \ref{int1} and \ref{eq:dual} holds. Moreover, one can prove that the existence of $u_0 \in X$ solution to \ref{int1} and $p_0 \in Y$ solution to \ref{eq:dual} is equivalent to the following \emph{optimality conditions}:
\begin{align}
    \left\{\begin{aligned}
K_* p_0 & \in \partial R\left(u_0\right), \\
K u_0 & =y_0.
\end{aligned}\right.
\end{align}
Therefore, if we find an element $\eta_0\in  \partial R\left(u_0\right)$ such that it satisfies $\eta_0=K_*p_0$, then it holds that $u_0$ is a solution to \ref{int1}. Hence,  following the definition proposed in \cite{duval2015exact}, we  call $\eta_0$ a dual certificate for $u_0$. Similar optimality conditions can be also obtained for \ref{eq:ip} and $\tilde \eta_{\lambda}$ denotes the dual certificate associated with $\tilde u_{\lambda}$, solution to \ref{eq:ip}. Since, in general, dual certificates for \ref{int1} are not unique, in the following we will consider the one with the minimal norm in $Y$.
\begin{defs}[Minimal-norm dual certificate]\label{def:dualcert}
   The minimal-norm dual certificate associated with \ref{int1} is defined as $\eta_0=K_* p_0$, where $p_0 \in Y$ is the unique solution to \ref{eq:dual} with minimal $\|\cdot\|_Y$ norm:
$$
p_0=\operatorname{argmin}\left\{\|p\|_Y: p \in Y \text { is a solution to } \mathcal{D}_h\left(y_0\right)\right\}.
$$
\end{defs}
\subsection{Non-degeneracy}
In \cite{duval2015exact}, it has been noted that for deconvolution problems in the space of Radon measures regularized with the total variation norm, a suitable non-degeneracy assumption on the dual certificate provides ESRR (or Exact Support Recovery recalling the denomination in \cite{duval2015exact}).
Such assumption is known as Non-Degeneracy Source Condition (NDSC). 
Given a sparse measure $\mu_0 = \sum_{i=1}^n c_0^i \delta_{x_0^i}$ such that $K\mu_0 = y_0$, the NDSC is divided in two parts. The first part, composed of the \emph{Source Condition} $(1)$ and the localization condition $(2)$, requires $\eta_0 \in C(\mathbb{T})$ to satisfy 
\begin{align}\label{eq:firsttwo}
    \operatorname{Im} K_* \cap \partial\left\|\mu_0\right\|_{M(\mathbb{T})} \neq \emptyset, \qquad \arg \max _x\left|\eta_0(x)\right|=\left\{x_0^1, \ldots, x_0^n\right\}.
\end{align}
Then, to promote uniqueness of the locations and the coefficients one also has to ensure a non-degeneracy on the second derivative of $\eta_0$, that is 
\begin{align}\label{eq:second}
\eta_0^{\prime \prime}\left(x_0^i\right) \neq 0, \quad \forall i=1, \ldots, n.
\end{align}
In \cite{carioni2023general}, we have extended the NDSC to general convex optimization problems by looking at how the dual certificate $\eta_0 \in X_*$ behaves when tested against extreme points of $B = \{u\in X : R(u) \leqslant 1\}$. In particular, we imposed conditions on the duality product 
\begin{align*}
    u \mapsto \langle \eta_0, u\rangle,
\end{align*}
when $u$ belongs to the extreme points of $B$. One can observe that in the case studied in \cite{duval2015exact} the extreme points of the total variation norm are Dirac deltas and thus the duality product $\langle \eta_0, u\rangle$ reduces to the pointwise evaluation of $\eta_0$. However, for general optimization problems this is not the case.
Moreover, in our setting is it not straightforward to impose a second order condition on $u \mapsto \langle \eta_0, u\rangle$ since the set of extreme points has a priori no differential structure. We thus had 
 to find a way to define a suitable non-degeneracy condition for $u \mapsto \langle \eta_0, u\rangle$ seen as a real-valued map from the metric space $\mathcal{B}:=\overline{\operatorname{Ext}(\{u \in X: R(u) \leqslant 1\})}^*$, with the metric $d_{\mathcal{B}}$ metrizing the weak* convergence on $\mathcal{B}$. This justifies the name Metric Non-Degenerate Source Condition (MNDSC).
Given $u_0 = \sum_{i=1}^n c_0^i u_0^i$ such that $Ku_0 = y_0$, the first two conditions \eqref{eq:firsttwo} are simply rewritten in our context using the duality product $\langle \eta_0, u\rangle$, that is
\begin{enumerate}
    \item[1)] $\operatorname{Im} K_* \cap \partial R\left(u_0\right) \neq \emptyset$;
    \item[2)] $\left\{u_0^1, \ldots, u_0^n\right\}=\operatorname{Exc}\left(u_0\right)$,
\end{enumerate}
where $\operatorname{Exc}\left(u_0\right):=\left\{u \in \mathcal{B}:\left\langle\eta_0, u\right\rangle=1\right\}$ is called \emph{extreme critical set} (it replaces the extended support in \cite{duval2015exact}). Instead, the third condition \eqref{eq:second} is rephrased using continuous curves in the metric space $\mathcal{B}$ that belongs to the following set for a fixed $M>0$:
\begin{align*}
\Gamma_M=\left\{\gamma \in C([0,1], \mathcal{B}): t \mapsto K(\gamma(t)) \text { is } C^2((0,1)) \text { and } \sup _{t \in[0,1]}\left\|\frac{d^2}{d t^2} K(\gamma(t))\right\|_Y \leqslant M\right\} .
\end{align*}
Precisely, we require the following non-degeneracy condition:
\begin{enumerate}
    \item[3)]  there exist $\varepsilon, \delta>0$ such that for any two elements in $B_{\varepsilon}(u_0^i):=\{u \in \mathcal{B}: d_{\mathcal{B}}(u_0^i, u) \leqslant \varepsilon\}$, there exists a curve $\gamma:[0,1] \rightarrow B_{\varepsilon}(u_0^i)$ belonging to $\Gamma_M$, connecting them, that satisfies:
\begin{align*}
    \frac{d^2}{d t^2}\left\langle\eta_0, \gamma(t)\right\rangle<-\delta \quad \forall t \in(0,1).
\end{align*}
\end{enumerate}
Let us notice that this condition must be satisfied at every point along such a curve, which is in line with the lack of differentiability of $\mathcal{B}$, making our condition not well-defined pointwise. The MNDSC comprises the conditions $1)$, $2)$ and $3)$. Let us remark that non-degeneracy assumptions have been used in the context of convex regularized optimization problems to prove convergence of infinite-dimensional sparse algorithms \cite{bredies2023asymptotic}. Our MNDSC can be seen as a version of it designed specifically for ESRR.

\section{Exact sparse representation recovery}
We now describe the main result in \cite{carioni2023general}, that we named, as mentioned above, \emph{Exact Sparse Representation Recovery} (ESRR). The name is justified since the solution  $\tilde u_{\lambda}$ to \ref{eq:ip} recovers the sparse representation of the original solution $u_0$ to \ref{int1}, with exactly the same number of extreme points. 
We consider the following set of admissible parameters/noise levels for $\lambda_0>0$ and $\alpha>0:$
\begin{align*}    
N_{\alpha, \lambda_0}=\left\{(\lambda, w) \in \mathbb{R}_{+} \times Y: 0 \leqslant \lambda \leqslant \lambda_0 \quad \text { and } \quad\|w\|_Y \leqslant \alpha \lambda\right\}.
\end{align*}
Then, the following result holds.
\begin{theorem}[ESRR]\label{ESRR}
 Let $u_0 = \sum_{i=1}^n c_0^i u_0^i$ with $c_0^i >0$ and 
$u_0^i \in \mathcal{B} \setminus \{0\}$ satisfy the MNDSC. Suppose that $\{Ku_0^i\}_{i=1}^n$
are linearly independent. Then, for $\varepsilon>0$ small enough, there exist suitable values $\alpha>0, \lambda_0>0$ such that, for every $(\lambda, w) \in N_{\alpha, \lambda_0}$, the solution $\tilde{u}_\lambda$ to \ref{eq:ip} is unique and admits a unique representation of the following form: 
\begin{align}
    \tilde{u}_\lambda=\sum_{i=1}^n \tilde{c}_\lambda^i \tilde{u}_\lambda^i,
\end{align}
where $\tilde{u}_\lambda^i \in B_{\varepsilon}(u_0^i) \backslash\{0\}$ such that $\left\langle\tilde{\eta}_\lambda, \tilde{u}_\lambda^i\right\rangle=1$,  $\tilde{c}_\lambda^i>0$ and $|\tilde c_\lambda^i - c_0^i| \leqslant \varepsilon$ for every $i = 1,\ldots,n$.
\end{theorem}
\noindent
 This result naturally extends the one in \cite{duval2015exact} with the only caveat that the decay rate $O(\lambda)$ obtained in \cite{duval2015exact} for the coefficients and the positions is not achieved in our scenario. This follows from the fact that since we assume no differentiability structure on $\mathcal{B}$, we cannot apply the implicit function theorem to both the coefficients and the locations as done in \cite{duval2015exact}. As a consequence, we cannot estimate the rate of change of the extreme points with respect to $\lambda$ and $w$, preventing us to prove the decay rate $O(\lambda)$.

\begin{rem}
  It is worth to notice that the geometry of $\mathcal{B}$ highly influences the MNDSC. In particular, it is immediate to see that if an extreme point of the representation of $u_0^i$ is isolated in $\mathcal{B}$, then condition 3) in the MNDSC is automatically satisfied. This links our theory to the finite dimensional case, with $R(u) = \|u\|_1$, where all extreme points are isolated and no non-degeneracy is required. Building on this remark, it is clear that a non-degeneracy needs to be imposed only along curves connecting extreme points that are arbitrarily close (in the weak* topology) to $u_0^i$. Depending on the geometry of $\mathcal{B}$, this fact can simplify the non-degeneracy that one needs to impose. 
\end{rem}
\noindent
In the next sections, we will apply our theorem to  two examples where the considerations of the previous remark will play a crucial role.


\section{Demixing spikes and kernels}

Here we consider the applications of Theorem \ref{ESRR} to a variant of the setting introduced in \cite{fernandez2016super}, c.f. Section 3.1. 
The aim is to reconstruct from a finite-dimensional measurement a signal defined on $\mathbb{T}$ that is composed by the superposition of convolution kernels and spikes, and to demix the kernel component from the spike component. 
In practice, we consider the following setting. Given
$X = M(\mathbb{T}, \mathbb{R}) \times \mathbb{R}^N$
that is the dual of $X_* = C(\mathbb{T}, \mathbb{R})\times  \mathbb{R}^N$, we define the  optimization problem:
\begin{align}\label{eq:spikes}
\min_{(\mu,\zeta) \in X} \frac{1}{2}\sum_{i=1}^N \Big|\int_{\mathbb{T}} \varphi_i(x) \, d\mu(x) + \zeta_i - (y_i+w_i)\Big|^2 + \lambda R(\mu,\zeta),
\end{align}
where $y =(y_1, \ldots,y_N) \in \mathbb{R}^N$ are the measurements, $w =(w_1, \ldots,w_N) \in \mathbb{R}^N$ is the noise, $\varphi_i : \mathbb{T} \rightarrow \mathbb{R}$ are the kernels of class $C^2(\mathbb{T})$, and the regularizer $R(\mu,\zeta)$ is defined as 
\begin{align*}
R(\mu,\zeta) = \|\mu\|_{TV} + \|\zeta\|_1,\quad  \mu\in  M(\mathbb{T}, \mathbb{R}), \zeta\in \mathbb{R}^N.
\end{align*}
Note that the definition of $R$ is aimed at promoting sparsity both in $\mu$ and $\zeta$.
Moreover, the terms $\int_{\mathbb{T}} \varphi_i(x) \, d\mu(x)$ are the kernel components while $\zeta_i$ are the spike components of the signal.
We define the operator $K:X \rightarrow \mathbb{R}^N$ as  
\begin{align*}
(K(\mu,\zeta))_i :=  \int_{\mathbb{T}} \varphi_i(x) \, d\mu(x) + \zeta_i, \qquad (\mu,\zeta) \in X,\ \  i= 1,\ldots,N.
\end{align*}
Now, we can rewrite problem \eqref{eq:spikes} as 
\begin{align*}
\min_{(\mu,\zeta) \in X} \frac{1}{2}\|K(\mu,\zeta) - y-w\|^2 + \lambda R(\mu,\zeta).
\end{align*}
It can be readily verified that the operator $K$ is weak*-to-strong continuous and the regularizer $R$ is convex, positively $1$-homogeneous and weak* lower-semicontinous with $0$ being an interior point of $\partial R(0)$. Moreover, the extreme points of the ball $B = \{(\mu,\zeta) \in X : R(\mu,\zeta) \leqslant1\}$ can be characterized as in the following proposition (see for example \cite[Section 4.2.3]{boyer2019representer}).
\begin{proposition}
It holds that 
\begin{equation}
{\rm Ext}(B) = E^+_1 \cup E^-_1\cup E_2,
\end{equation}
where
\begin{align*}
E^+_1 := \{ (\delta_x,0) : x \in \mathbb{T}\},\quad E^-_1 := \{ (-\delta_x,0) : x \in \mathbb{T}\}, \quad E_2 := \{(0,ae_k) :  k =1, \ldots, N, a\in \mathbb{R}, |a| =1 \},
\end{align*}
where $e_k$ is the k-th vector of the canonical base of $\mathbb{R}^N$.
\end{proposition}
\noindent
It is also straightforward to verify that ${\rm Ext}(B)$ is weak* closed since it is the union of three weak* closed sets, one of which is finite dimensional.
Moreover, a fundamental topological property of this set is that it can be written as the union of isolated components (according to the weak* topology). Such observation is provided in the following remark.

\begin{rem}\label{rem:isolated}
The set $E_2$ is made of isolated points. Moreover, given $\delta_x \in E^+_1$ there exists a ball $B$ in the weak* topology that contains $\delta_x$ and such that $B \cap E^-_1 \cap E_2 = \emptyset$ and similarly given $-\delta_x \in E^-_1$ there exists a ball $B$ in the weak* topology that contains $-\delta_x$ and such that $B \cap E^+_1 \cap E_2 = \emptyset$. 
\end{rem}
\noindent
In the following result we denote by $\eta_0 = (\eta_0^1, \eta_0^2) \in C(\mathbb{T}, \mathbb{R}) \times \mathbb{R}^N$ the minimal-norm dual certificate for \eqref{eq:spikes}, defined as in Section \ref{def:dualcert}. We are now ready to state the exact sparse representation recovery result for the optimization problem \eqref{eq:spikes}.

\begin{theorem}\label{thm:spikes}
    Let $\mu_0 = \sum_{i=1}^n c_0^i (\sigma_0^i \delta_{x_0^i}, 0) + \sum_{j=1}^m d_0^j (0, \zeta_0^j)$ be a linear combination of extreme points, where $c_0^i, d_0^j \in \mathbb{R}\setminus \{0\}$, $\sigma_0^i \in \{-1,+1\}$, $x_0^i \in \mathbb{T}$ and $(0,\zeta_0^j) \in E_2$. Suppose that 
\begin{enumerate}
    \item [a)] ${\rm Im} K_* \cap \partial R(\mu_0) \neq \emptyset$;
    \item[b)]  $\sigma_i \eta^1_0(x) = 1$ if and only if $x = x_0^i$ for $i=1,\ldots,n$, and $\langle \eta_0^2, \zeta \rangle = 1$ if and only if $\zeta = \zeta_0^j$ for $j=1,\ldots,m$;
\item[c)] $\eta_0^1 \in C^2(\mathbb{T})$ is such that $\sigma_i (\eta^1_0)''(x_0^i) < 0$ for $i=1,\ldots,n$.
\end{enumerate}
Additionally, suppose that $\{K(\sigma_i \delta_{x_0^i}, 0), K(0,\xi_0^j)\}_{i,j}$ are linearly independent. Then, for every sufficiently small $\varepsilon >0$ there exist $\alpha >0$ and $\lambda_0>0$ such that, for every $(\lambda,w) \in N_{\alpha,\lambda_0}$, the solution $\tilde \mu_\lambda$ to \eqref{eq:spikes} is unique and it admits a unique representation of the form:
\begin{align}
    \tilde \mu_\lambda = \sum_{i=1}^n \tilde c_\lambda^i (\sigma_i \delta_{\tilde x_\lambda^i}, 0) + \sum_{j=1}^m \tilde  d_\lambda^j (0, \zeta_0^j)  
\end{align}
where $|\tilde x_\lambda^i - x_0^i| \leqslant \varepsilon$, $|\tilde c_\lambda^i -  c_0^i| \leqslant \varepsilon$ and $|\tilde d_\lambda^j - d_0^j| \leqslant \varepsilon$ for $i=1,\ldots, n$ and $j=1,\ldots, m$. 
\end{theorem}
\begin{proof}
    The proof follows from the application of Theorem \ref{ESRR}. In particular, we note that assumptions $a)$ and $b)$ correspond to $1)$ and $2)$ of the MNDSC. 
    Now, we show that $c)$ implies $3)$. The extreme points in the representation of $\mu_0$ are of the form either $(\sigma_i \delta_{x_0^i}, 0)$ or $(0, \zeta_0^j)$. For simplicity, let us call $(\sigma \delta_{x_0}, 0)$ an extreme point of the first type and $(0, \zeta_0)$ an extreme point of the second type. If $\varepsilon$ is sufficiently small, due to Remark \ref{rem:isolated}, we can ensure that there are no other extreme points in a ball $B_\varepsilon(0,  \zeta_0)$. In particular, we do not have to verify condition $3)$ in the MNDSC for such extreme points.
    Similarly, again due to Remark \ref{rem:isolated}, all extreme points in a ball $B_\varepsilon(\sigma \delta_{x_0}, 0)$ are of the form $(\sigma \delta_{x}, 0)$, where $|x-x_0| \leqslant \varepsilon$. 
    Given two extreme points in such a ball, denoted by $(\sigma \delta_{x_1}, 0)$ and $(\sigma \delta_{x_2}, 0)$, we can construct a curve $\gamma : [0,1] \rightarrow B_\varepsilon(\sigma \delta_{x_0}, 0)$ such that $\gamma(0) =  (\sigma \delta_{x_1}, 0)$ and $\gamma(1) =  (\sigma \delta_{x_2}, 0)$ as
    \begin{align}
        \gamma(t) = (\sigma \delta_{tx_2 + (1-t)x_1}, 0).
    \end{align}
Therefore, 
\begin{align}
    \langle \eta_0, \gamma(t) \rangle = \sigma \eta_0^1(tx_2 + (1-t)x_1).
\end{align}
By choosing $\varepsilon$ small enough and applying $c)$, there exists $\delta >0$ such that for all $t \in (0,1)$, the following holds: 
\begin{align}
    \frac{d^2}{dt^2} \langle \eta_0, \gamma(t) \rangle = \sigma\frac{d^2\eta^1_0}{dt^2}(tx_2 + (1-t)x_1)(x_2 - x_1)^2 < -\delta. 
\end{align}
This shows that $3)$ in the MNDSC holds. We can thus apply Theorem \ref{ESRR} and conclude the proof.
\end{proof}
\begin{rem}
   Note that the condition $b)$ in Theorem \ref{thm:spikes} involving $\eta_0^2$ is simply the classical sufficient condition for identifiability in LASSO.  
   Moreover, in condition $c)$, only $\eta_0^1$ is considered. This, as anticipated in the previous section, is due to the fact that the extreme points of the set $E_2$ are isolated (in the weak* topology), which makes them already satisfying condition $3)$ of the MNDSC. 
\end{rem}

\section{Group BLASSO}

As a second application we consider the problem of reconstructing a vector valued measure on the torus from a finite number of noisy convolution measurements. The ESRR for TV-regularized vector measures has been already considered in \cite{golbabaee2022off}. Here, we show that Theorem \ref{ESRR} applies straightforwardly to this setting and it allows to prove ESRR under a slightly different non-degeneracy condition than \cite{golbabaee2022off}.
\\
We denote the space of such measures $X= M(\mathbb{T}; \mathbb{R}^d)$. Note that the predual space of $X$ is $X_* = C(\mathbb{T}; \mathbb{R}^d)$. We define the following regularized optimization problem:
\begin{align}\label{eq:group}
\min_{\mu\in X} \frac{1}{2}\sum_{i=1}^N \Big|\int_{\mathbb{T}} \varphi_i(x) \cdot d\mu(x) - (y_i+w_i)\Big|^2 + \lambda R(\mu),
\end{align}
where $y =(y_1, \ldots,y_N) \in \mathbb{R}^N$ are the measurements, $w =(w_1, \ldots,w_N) \in \mathbb{R}^N$ is the noise, $\varphi_i \in C^2(\mathbb{T}, \mathbb{R}^d)$ are the kernels, and the regularizer $R$ is defined as 
\begin{align}
    R(\mu) = \|\mu\|_{TV} = \sup\left\{ \int_{\mathbb{T}} \psi(x)\cdot \, d\mu(x) : \psi \in C(\mathbb{T}; \mathbb{R}^d),\ \sup_{x\in\mathbb{T}} \|\psi(x)\|_2 \leqslant1\right\}.
\end{align}
Note that $R$ is the classical total variation norm of vector measures, where we endowed $\mathbb{R}^d$ with the $2$-norm.
If we define the operator $K:X \rightarrow \mathbb{R}^N$ as  
\begin{align*}
(K\mu)_i =  \int_{\mathbb{T}} \varphi_i(x) d\mu(x), \qquad \mu \in X,\ \  i= 1,\ldots,N,
\end{align*}
then we can rewrite the optimization problem \eqref{eq:group} as 
\begin{align*}
\min_{\mu \in X} \frac{1}{2}\|K\mu - y-w\|^2 + \lambda R(\mu).
\end{align*}
Also in this scenario it can be readily verified that the operator $K$ is weak*-to-strong continuous and the regularizer $R$ is convex, positively $1$-homogeneous, weak* lower-semicontinous and $0$ is an interior point of $\partial R(0)$. 
The extreme points of the ball $B = \{\mu \in X : R(\mu) \leqslant1\}$ is a weak* closed set and it can be characterized as follows (see for instance \cite[Section 4.2.3]{boyer2019representer}).
\begin{proposition}
It holds that 
\begin{equation}
{\rm Ext}(B) = \{a \delta_x: x\in \mathbb{T}, a \in \mathbb{R}^d,\, \|a\|_2 = 1\}.
\end{equation}
\end{proposition}

\begin{theorem}\label{thm:spikes2}
    Let $\mu_0 = \sum_{i=1}^n c_0^i a_0^i \delta_{x_0^i}$ be a linear combination of extreme points, where $c_0^i \in \mathbb{R}\setminus \{0\}$, $a_0^i \in \mathbb{R}^d$ with $\|a_0^i\|_2 =1$ and $x_0^i \in \mathbb{T}$. Suppose that 
\begin{enumerate}
    \item [a)] ${\rm Im} K_* \cap \partial R(\mu_0) \neq \emptyset$;
    \item[b)]  $\langle \eta_0(x), a\rangle = 1$ if and only if $x = x_0^i$, $a = a_0^i$ for $i=1,\ldots,n$;
\item[c)] $\eta_0 \in C^2(\mathbb{T})$ is such that $\langle \eta''_0(x_0^i), a_0^i\rangle < 0$ for $i=1,\ldots,n$.
\end{enumerate}
Additionally, suppose that $\{K(a_0^i\delta_{x_0^i}))\}_{i=1}^n$ are linearly independent. Then, for every sufficiently small $\varepsilon >0$ there exist $\alpha >0$ and $\lambda_0>0$ such that, for every $(\lambda,w) \in N_{\alpha,\lambda_0}$, the solution $\tilde \mu_\lambda$ to \eqref{eq:group} is unique and it admits a unique representation of the form:
\begin{align}
    \tilde \mu_\lambda = \sum_{i=1}^n \tilde c_\lambda^i \tilde a^i_\lambda \delta_{\tilde x_\lambda^i} \quad \tilde x_\lambda^i \in \mathbb{T},\ \|\tilde a^i_\lambda\|_2 = 1,\, c_\lambda^i>0,  
\end{align}
where $|\tilde x_\lambda^i - x_0^i| \leqslant \varepsilon$, $\|\tilde a_\lambda^i - a_0^i\|_2 \leqslant \varepsilon$, and $|\tilde c_\lambda^i - c_0^i| \leqslant \varepsilon$ for $i=1,\ldots, n$.
\end{theorem}
\begin{proof}
 The proof follows again from the application of Theorem \ref{ESRR}. In particular, we note that assumptions $a)$ and $b)$ correspond to $1)$ and $2)$ of the MNDSC. Now, we show that $c)$ implies $3)$. Consider any extreme point in the representation of $\mu_0$ and denote it by 
    $a_0 \delta_{x_0}$. If $\varepsilon$ is sufficiently small, we can ensure that all extreme points in a ball $B_\varepsilon(a_0 \delta_{x_0})$ are of the form $a\delta_{x}$, where $|x-x_0| \leqslant \varepsilon$ and $\|a-a_0\|_2 \leqslant \varepsilon$.  Given two extreme points in such a ball, denoted by $a_1 \delta_{x_1}$ and $a_2 \delta_{x_2}$, we can construct a curve $\gamma : [0,1] \rightarrow B_\varepsilon(a_0 \delta_{x_0})$ such that $\gamma(0) =  a_1 \delta_{x_1}$ and $\gamma(1) = a_2 \delta_{x_2}$ as
    \begin{align}
        \gamma(t) = \frac{ta_2 + (1-t) a_1}{\|ta_2 + (1-t) a_1\|_2}\delta_{tx_2 + (1-t)x_1}.
    \end{align}
Therefore, 
\begin{align}
    \langle \eta_0, \gamma(t) \rangle = \langle \frac{ta_2 + (1-t) a_1}{\|ta_2 + (1-t) a_1\|_2}, \eta_0(tx_2 + (1-t)x_1)\rangle.
\end{align}
We now compute the second derivative of $t \mapsto \langle \eta_0, \gamma(t) \rangle$, where for simplicity we set $a_t = ta_2 + (1-t) a_1$ and $x_t = tx_2 + (1-t)x_1$. Note that $\langle \eta_0, \gamma(t) \rangle$ is twice differentiable, since by choosing $\varepsilon$ small enough, it holds that $\|a_t\|_2 \neq 0$ for all $t \in (0,1)$ and $\eta_0 \in C^2(\mathbb{T})$ by assumption. The first derivative and the second derivative are equal to 
\begin{align}
    \frac{d}{dt} \langle \eta_0, \gamma(t) \rangle = \langle \frac{S_t}{\|a_t\|_2^3}, \eta_0(x_t)\rangle + \langle \frac{a_t}{\|a_t\|_2}, \eta_0'(x_t)(x_2 - x_1)\rangle,
\end{align}
\begin{align*}
    & \frac{d^2}{dt^2} \langle \eta_0, \gamma(t) \rangle = 2 \langle \frac{S_t}{\|a_t\|_2^3}, \eta'_0(x_t)(x_2 - x_1)\rangle + \langle \frac{a_t}{\|a_t\|_2}, \eta_0''(x_t)(x_2 - x_1)^2\rangle\\
    & + \frac{\langle a_t, a_2 - a_1\rangle (a_2 - a_1) - \|a_2 - a_1\|_2^2 a_t - 3\|a_t\|_2 \langle a_t,a_2 - a_1\rangle S_t}{\|a_t\|_2^6},
\end{align*}
where $S_t = \|a_t\|_2^2(a_2 - a_1) - \langle a_t, a_2 - a_1\rangle a_t$. Therefore, by choosing $\varepsilon$ small enough and applying $c)$, there exists $\delta > 0$  such that  
$\frac{d^2}{d t^2}\left\langle\eta_0, \gamma(t)\right\rangle< -\delta$ for all $t\in (0,1)$.
This shows that $3)$ in the MNDSC holds. We can thus apply Theorem \ref{ESRR} and conclude the proof.
\end{proof}

\begin{rem}
An even simpler argument leads to ESRR if instead of the 2-norm, we consider the 1-norm. In this case, the extreme points can be characterized as follows:
\begin{equation}
{\rm Ext}(B_1) = \{ae_k \delta_x: x\in \mathbb{T}, k = 1 \ldots, d, a\in\mathbb{R}, |a| = 1\},
\end{equation}
where $e_k$ is the k-th vector of the canonical base of $\mathbb{R}^d$. In this case the set ${\rm Ext}(B_1)$ can be decomposed in sets $E_{a,k} = \{ae_k\delta_x : x \in \mathbb{T}\}$
that are well-separated according to the weak* topology, meaning that for each point in $E_{a,k}$ there exists a weak* ball that does not intersect the other sets. Therefore, a curve in ${\rm Ext}(B_1)$ connecting any pair of extreme points in a sufficiently small weak* ball can be just constructed as $a_0 \delta_{tx_2 + (1-t)x_1}$, for a constant coefficient $a_0$. As a consequence, a simple computation, involving only the derivative of the locations, shows that a similar MNDSC to the one of Theorem \ref{thm:spikes2} leads to ESRR.
\end{rem}

 \bibliographystyle{plain}
 \bibliography{pamm-tpl.bib}

\end{document}